\documentclass[10pt]{amsart}
\usepackage{bbm}
\usepackage{amsfonts}
\usepackage{amscd}
\usepackage{cases}
\usepackage{amssymb}
\usepackage{amsmath,amssymb,amsfonts,amsthm,graphics,
latexsym, amscd, amsfonts, epsfig, eepic,epic,psfrag}
\usepackage{mathrsfs}
\usepackage{color}
\usepackage{eucal}
\usepackage{eufrak}
\usepackage[all]{xypic}
\usepackage{xspace}

\theoremstyle{plain}
\newtheorem{theorem}{Theorem}
\newtheorem{corollary}{Corollary}

\newtheorem{lemma}{Lemma}

\theoremstyle{definition}
\newtheorem{definition}{Definition}

\theoremstyle{example}

\theoremstyle{remark}

\numberwithin{equation}{section}

\begin{document}
\begin{center}
{\bf\Large A bijection between unicellular and bicellular maps}
\\
\vspace{15pt} Hillary S.~W. Han and Christian M. Reidys$^{\,\star}$
\end{center}

\begin{center}
         Department of Mathematics and Computer Science  \\
         University of Southern Denmark, Campusvej 55, \\
         DK-5230, Odense M, Denmark \\
         Phone$^{\,\star}$: 45-24409251 \\
         Fax$^{\,\star}$: 45-65502325 \\
         email$^{\,\star}$: duck@santafe.edu \\

\end{center}

\centerline{\bf Abstract}
In this paper we present a combinatorial proof of a relation between
the generating functions of unicellular and bicellular maps. 
This relation is a consequence of the Schwinger-Dyson equation of 
matrix theory. Alternatively it can be proved using representation 
theory of the symmetric group.
Here we give a bijective proof by rewiring unicellular maps
of topological genus $(g+1)$ into bicellular maps of genus $g$ and 
pairs of unicellular maps of lower topological genera. 
Our result has immediate consequences for the folding of RNA interaction 
structures, since the time complexity of folding the transformed structure 
is $O((n+m)^5)$, where $n,m$ are the lengths of the respective backbones,
while the folding of the original structure has $O(n^6)$ time 
complexity, where $n$ is the length of the longer sequence.


\section{Introduction}

In this paper we present a combinatorial proof of a relation between
the generating functions of unicellular and bicellular maps,
${\bf C}_{g}(z)$ and ${\bf C}_{g}^{[2]}(z)$:
\begin{equation}\label{E:1}
\sum_{g_1=0}^{g+1}\,{\bf C}_{g_1}(z){\bf C}_{g+1-g_1}(z)
+{\bf C}_{g}^{[2]}(z)= {\bf C}_{g+1}(z)/z.
\end{equation}

Eq.~(\ref{E:1}) is a consequence of the Schwinger-Dyson equation of 
matrix theory \cite{Dyson,Schwinger}. It can also be proved by extending 
the representation theoretic framework of Zagier \cite{Zagier}. To the best 
of our knowledge, our bijection represents the first combinatorial proof of 
eq.~(\ref{E:1}).

The motivation for this paper stems from the algorithmic folding problem
of RNA-pseudoknot structures over one and two backbones \cite{gfold,fenix2bb}.
The folding of RNA molecules means to identify some minimum energy 
configuration of a given sequence. These configurations are subject to 
certain constraints on how two nucleotides can bond 
\cite{gfold,fenix2bb,Fenix:08}.
RNA structures over two backbones are called RNA-RNA interaction structures 
\cite{rip1, rip2} and of importance in the context of many biochemical, 
regulatory activities. 

Theorem~\ref{T:thm1} provides a rewiring algorithm transforming bicellular 
maps into certain unicellular maps. It thus allows to reduce the folding 
problem of RNA-RNA interaction structures to that of 
RNA-pseudoknot structures over one backbone, see Fig.~\ref{F:2b-1b}. 
This rewiring is of practical interest, since the time 
complexity of folding the rewired interaction structure is given by 
$O((n+m)^5)$, where $n,m$ are the lengths of the respective backbones.
The direct folding of the interaction structure however has a time 
complexity of $O(n^6)$ where $n$ is the length of the longer sequence.
Since there exist an abundance of ``small RNA'' interactions between 
a large and a very small RNA structure, the $O((n+m)^5)$ time complexity
is oftentimes much smaller than $O(n^6)$ \cite{fenix2bb}.

\begin{figure}[ht]
\begin{center}
\includegraphics[width=0.9\columnwidth]{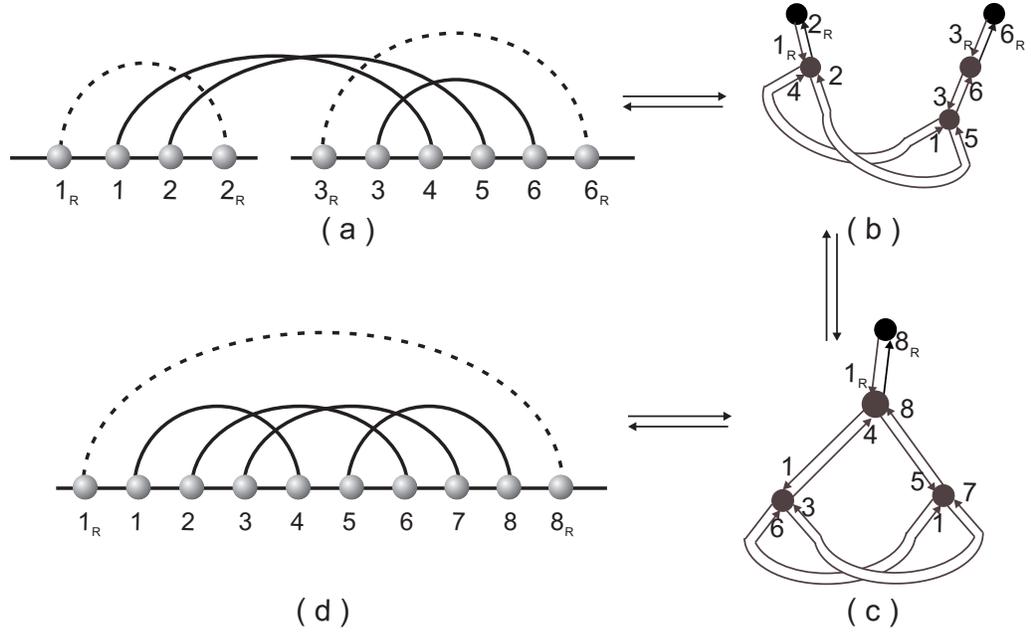}
\end{center}
\caption{\small Rewriting diagram over 2 backbones into diagrams over one
backbone. A diagram over 2 backbones $(a)$, its bicellular map $(b)$
and its corresponding unicellular map $(c)$ and its diagram over one backbone 
$(d)$.}\label{F:2b-1b}
\end{figure}

The paper is organized as follows: first we recall some basic facts
about diagrams, fatgraphs, unicellular and bicellular maps. We shall work
with planted unicellular and double planted bicellular maps. These plants
are additional vertices of degree one and emerge naturally in the context 
of RNA as there is a $5^{'}$ to $3^{'}$ orientation of the molecular backbone. 
Modulo Poincar\'{e}-duality the plant ``marks'' the beginning and ending of 
the backbone. Interestingly, the plants themselves play a key role in the 
combinatorial construction.

Second we dissect the bijection into three separate maps by introducing a certain
partition of the set of unicellular maps of fixed topological genus. Then we 
prove our two main lemmas. These show that, with respect to the above mentioned 
partition, a unicellular map either corresponds uniquely to a pair of unicellular 
maps of lower genus (Lemma~\ref{L:lem1}) or to a unique bicellular map of lower genus
(Lemma~\ref{L:lem2}).

We third prove the main result, Theorem~\ref{T:thm1}, by combining the two 
lemmas and give eq.~(\ref{E:1}) as an enumerative corollary.

\section{Some basic facts}

\subsection{Unicellular maps and bicellular maps}

\begin{definition}
A unicellular map, $u$ with $n$ edges is a triple $u=([2n],\alpha,\sigma)$, where
$\alpha$ is an involution of $[2n]$ without fixed points and $\sigma$ is a
permutation of $[2n]$ such that $\gamma=\alpha\circ\sigma $ has only one cycle.
The elements of $[2n]$ are called half-edges of $u$. The cycles of $\alpha$ and
$\sigma$ are called the edges and the vertices of $u$, respectively. The
permutation $\gamma$ is called the face or boundary component of $u$.
\end{definition}

Given a unicellular map $u=([2n], \alpha, \sigma)$, its associated graph $G$ is
the graph whose edges are given by the cycles of $\alpha$, vertices by the cycles
of $\sigma$.
We can consider a $G$-edge as a ribbon whose two sides are labeled by the
half-edges as follows:
if a half-edge $h$ belongs to a cycle $e$ of $\alpha$ and a certain $v$ of
$\sigma$, then $h$ is the right-hand side of the ribbon corresponding to $e$,
when entering $v$.

We draw the graph $G$ in such a way that around each vertex $v$, the counterclockwise
ordering of the half-edges belonging to the cycle $v$ is given by the cycle $v$. This
ordering of half-edges enriches the combinatorial graph $G$ to a ribbon graph or fat
graph $\mathbb{G}$. Clearly, a fat graph $\mathbb{G}$ with one boundary component is
tantamount to the unicellular map $u$, see Fig.~\ref{F:unicellular}(a).
$\gamma=\alpha\circ\sigma $ is interpreted as the cycle of half-edges visited 
when making the tour of the graph, keeping the graph on its left.

\begin{definition}
A planted unicellular map having $n$ edges is a unicellular map
$u=([2n+2],\alpha,\sigma)$, such that $(1,2n+2)$ is a cycle of
$\alpha$. We shall label the face of $u$ as
$$
\gamma=[1_R,1, 2, \ldots, 2n, 2n_R]
$$
and denote $(2n_R)$ as $p$, the plant of $u$.
\end{definition}

Given a planted unicellular map $u$ the face $\gamma$ induces a linear order
$<_{u}$ on $H$ via:
$$
1_R <_{u} \gamma(1_R)<_{u} \gamma^2(1_R)<_{u} \ldots <_{u} \gamma^{2n-1}(1_R)
<_{u} \gamma^{2n}(1_R)<\gamma^{2n+1}(1_R)=2n_R.
$$

\begin{figure}[ht]
\begin{center}
\includegraphics[width=0.8\columnwidth]{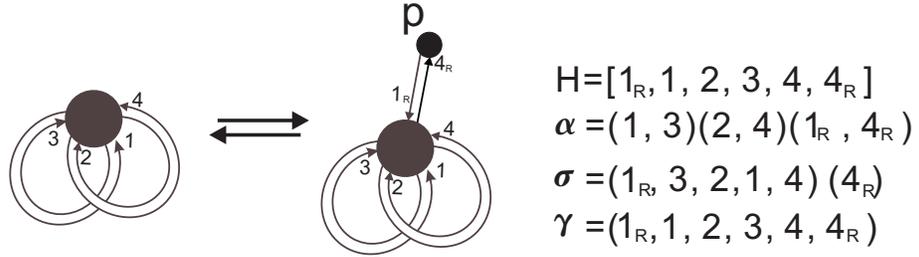}
\end{center}
\caption{\small A unicellular map with one vertex without (left) and with 
plant (right).
The half-edges of a vertex are read counterclockwise, i.e.,
$(1_R,3,2,1,4)$. The vertex $(4_R)$ is the plant.
}\label{F:unicellular}
\end{figure}

Suppose $u$ has $J$ vertices, $v_1, v_2, \ldots, v_J$. Then there is a natural
equivalence relation of half-edges, given by $h \sim \alpha(h)$ and in particular, 
$1_R \sim 2n_R$.

For each vertex $v_j$, $j \in J$, let $\min_u(v_j)$ denote the first half-edge 
where $\gamma$ arrives at $v_j$. We write $v_j$, reading the 
$v_j$-half-edges counter clockwise and starting at $h_j^1=\min(v_j)$:
$$
v_j=(h_j^1, \ldots, h_j^{m_j}).
$$
In particular, the vertex containing the half-edge $1_R$ is denoted by $v_1$. 
The order $<_u$ induces thus a linear order $<_v$ on the vertices by setting 
$v_i<v_j$ iff $\min_u(v_i) < \min_u(v_j)$.


\begin{definition}
A planted bicellular map $b$ having $n$ edges is a triple $b=(L, \beta, \tau)$,
where $L$ is a set of cardinality $(2n+4)$ such that
$$
L=\{1_R, 1, \ldots, m, m_R, (m+1)_R,  m+1, \ldots,  2n,  2n_R\},  \qquad
1< m < 2n-1.
$$
$\beta$ is a fixed-point free involution containing the cycles $(1_R,m_R)$ and
$((m+1)_R, 2n_R)$. $\beta\circ\tau$ consists of the two cycles
$$
\omega_1=(1_R, 1, 2, \ldots, m, m_R), \quad \text{\rm } \quad
\omega_2=((m+1)_R,  m+1,  m+2, \ldots,  2n,  2n_R).
$$
The elements of $[2n]$ are called half-edges of $b$ and there exists some 
half-edge $x \in \omega_1$, such that $\beta(x) \in \omega_2$.
\end{definition}
The cycles $\beta \setminus \{(1_R,m_R) \cup ((m+1)_R, 2n_R)\}$ and $\tau \setminus
\{m_R,(2n_R)\}$ are called the {edges} and {vertices} of $b$. $\omega_1$ and
$\omega_2$ are the two {faces} of $b$. The cycles $p_1=(m_R)$ and $p_2=(2n_R)$ are
the two plants, see Fig.\ref{F:bi-example}.
We furthermore assume the following linear order of the half-edges of the two faces
$\omega_1$ and $\omega_2$:
$$
1_R <_{b} \omega_1(1_R)<_{b} \ldots  \omega_1^{m+1}(1_R)=m_R<_{b} (m+1)_R <_{b}
\omega_2((m+1)_R)<_{b} \ldots  \omega_2^{2n-m+1}((m+1)_R)=2n_R.
$$

\begin{figure}[ht]
\centerline{%
\epsfig{file=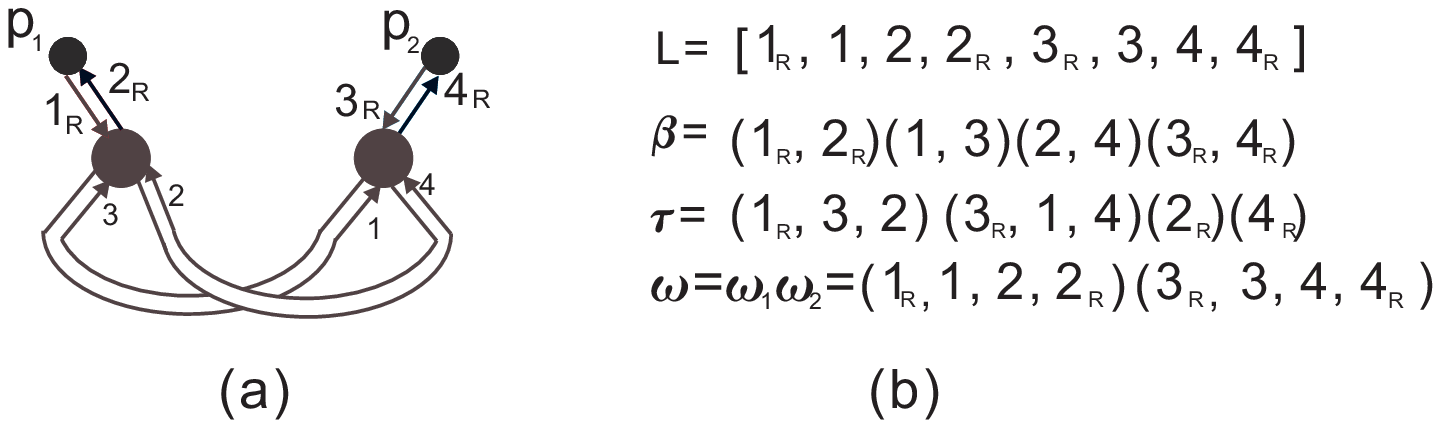,width=0.8 \textwidth}\hskip15pt }
\caption{\small A bicellular map with $2$ edges, $2$ vertices, and genus $0$ as:
(a) ribbon graph; (b) pair of permutations;.
}\label{F:bi-example}
\end{figure}

As in the case of unicellular maps, a bicellular map, $b=(L, \beta, \tau)$, has an
associated connected graph $G^{'}$, whose edges are given by the cycles of 
$\beta \setminus\{(1_R,m_R) \cup ((m+1)_R, 2n_R)\}$, vertices by the cycles of 
$\tau \setminus\{(m)_R,(2n_R)\}$. $G'$ can also be fattened to $\mathbb{G}'$
and the notion of $\min_b(v_j)$ as well as the linear order $<_{v, b}$ on the vertices of 
the bicellular map are defined  analogously.


\subsection{The RNA connection}


RNA molecules are linear biopolymers consisting of the four nucleotides $A$, $U$, $C$,
and $G$ characterized by a sequence endowed with a unique orientation 
($5^{'}$ to $3^{'}$).
Each nucleotide can interact (base pair) with at most one other nucleotide by
means of specific hydrogen bonds. Only the Watson-Crick pairs $GC$
and $AU$ as well as the wobble $GU$ are admissible. 
RNA structures
can be presented  as {diagrams}, that is a structure with a labeled graph $G$ over the
set $[N]=\{1, 2, \ldots, N \}$  represented  by drawing the vertices
$1, 2, \ldots, N$ on a horizontal line in the natural order and the arcs
$(i, j)$, where $i < j$, in the upper half-plane, see Fig. ~\ref{F:diagramrepre}.
A backbone is a sequence of consecutive integers contained in $[N]$.
A diagram over $b$ backbones is a diagram together with a partition of $[N]$ into $b$
backbones. The cases $b=1$ and $b=2$ are referred to as RNA structures and RNA
interaction structures.
\begin{figure}[ht]
\begin{center}
\includegraphics[width=0.8\columnwidth]{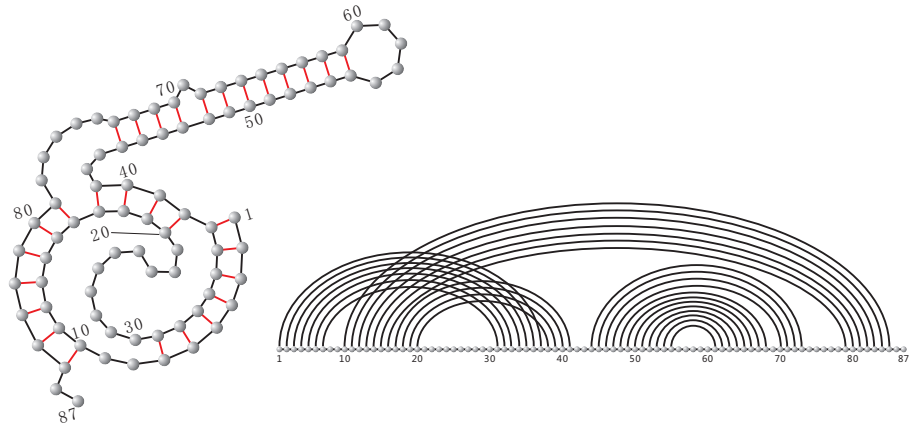}
\end{center}
\caption{\small An RNA structure as a planar graph and as a 
          diagram}\label{F:diagramrepre}
\end{figure}

RNA structures and interaction structures contain more information that just the
set of contacts between nucleotides. Aside form the $5^{'}$ to $3^{'}$ orientation of
the backbone itself there is in addition a fixed ordering of the backbone relative to
the base pairs.
This orientation implies that the contact graph together with the backbone gives rise
to a natural fatgraph structure as shown in Fig.~\ref{F:bc-one}, $\mathbb{G}$
\cite{Loebl:08,Penner:10}. We obtain again the counterclockwise traveling of the 
half-edges around each vertex as for unicellular maps. This fattening works 
analogously for RNA diagrams over two backbones \cite{Reidys:top1,fenix2bb}, 
see Fig.~\ref{F:bc-two}.

Euler's characteristic equation shows that, without affecting the topological type 
of the fatgraph $\mathbb{G}$, one can collapse each backbone into a single vertex 
with the induced fattening. In other words, there is an equivalent fatgraph 
representation of RNA-diagrams having a vertex for each respective backbone. 
Moreover, we may enrich this representation by adding an arc that labels the 
$5^{'}$ to $3^{'}$ end of the backbone. We refer to this arc as rainbow-arc or just 
rainbow.

\begin{figure}[ht]
\begin{center}
\includegraphics[width=0.9\columnwidth]{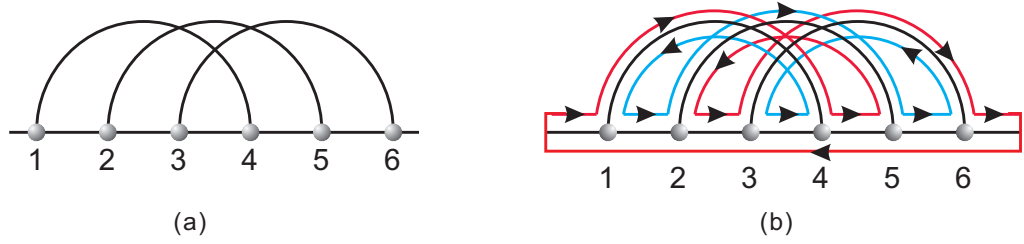}
\end{center}
\caption{\small A diagram with $3$ arcs $(a)$ is fattened $(b)$. $\gamma=\alpha \circ \sigma=
(4,2,6)(1,5,3)$ are its two boundary components (red)(blue).}\label{F:bc-one}
\end{figure}

\begin{figure}[ht]
\begin{center}
\includegraphics[width=0.7\columnwidth]{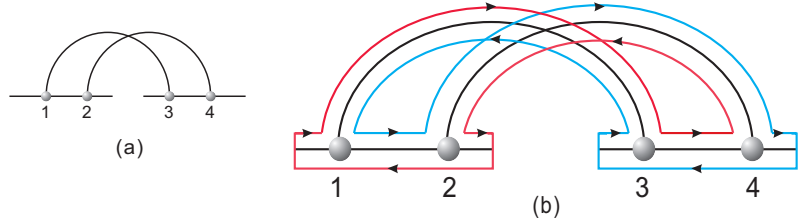}
\end{center}
\caption{\small A diagram with $2$ arcs over two backbones $(a)$ and its fattening $(b)$.
$\omega=\beta \circ \tau=(3, 2)(1, 4)$ are its two 
boundary components (red) (blue).}\label{F:bc-two}
\end{figure}

Clearly the $n$ arcs of the diagram determine after fattening $2n$ halfedges and 
the fatgraph consists of a pair $(\alpha',\sigma')$ together with the additional 
rainbow arc. Then, the mapping
$$
(\sigma',\alpha')\mapsto (\alpha'\circ \sigma',\alpha')
$$
is a bijection mapping vertices into boundary components. Topologically this is the 
Poincar\'{e} dual, mapping a fatgraph over one-backbone with rainbow into an planted 
unicellular map, see Fig.~\ref{F:dual}.

\begin{figure}[ht]
\begin{center}
\includegraphics[width=1.0\columnwidth]{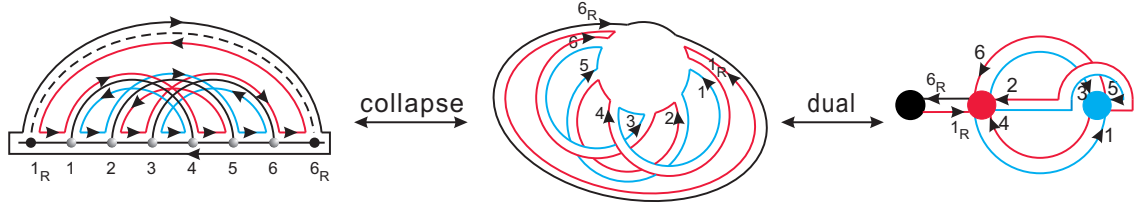}
\end{center}
\caption{\small A fattened diagram over one backbone, backbone collapse and 
the resulting planted, unicellular map.}\label{F:dual}
\end{figure}

The scenario is analogus for RNA-diagrams over two backbones, where we insert two 
rainbows over the respective backbones. Formation of the Poincar\'{e} dual, as 
illustrated in Fig.~\ref{F:dual2} generates a planted bicellular map.

\begin{figure}[ht]
\begin{center}
\includegraphics[width=1.0\columnwidth]{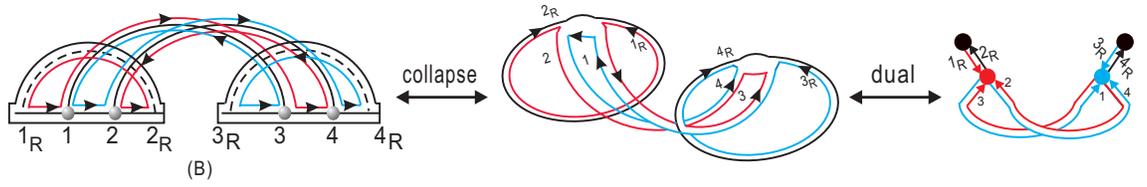}
\end{center}
\caption{\small A fattened diagram over two backbones, backbone collapse and 
the resulting planted, bicellular map. }\label{F:dual2}
\end{figure}

\section{Two lemmas}

Let $U_{g+1,n+1}$ denote the set of unicellular maps of genus $(g+1)$ with $(n+1)$ edges.
We observe that $U_{g+1,n+1}$ partitions into the following three classes
\begin{equation}\label{E:U}
U_{g+1,n+1}=U^I_{g+1,n+1}  \dot\cup\;
U^{II}_{g+1,n+1} \dot\cup \; U^{III}_{g+1,n+1},
\end{equation}
where
\begin{itemize}
\item $U_{g+1,n+1}^{I}$,
      the set of unicellular maps of genus $(g+1)$ in which $1$ and $\alpha(1)$
      are incident to two different vertices, such that
$$
\exists \; k;\  1 <  k < \alpha(1); \qquad  \alpha(1) < \alpha(k).
$$
\item $U_{g+1,n+1}^{II}$,
      the set of unicellular maps of genus $(g+1)$ in which $1$ and
      $\alpha(1)$ belong to the same vertex,
$$
\exists \; k;\  1 <  k < \alpha(1); \qquad  \alpha(1) < \alpha(k).
$$      
\item $U_{g+1,n+1}^{III}$,
      the set of unicellular maps of genus $g$ in which $1$ and $\alpha(1)$
      are incident to two different vertices, such that
$$
\forall \;k;\  1 < k <  \alpha(1); \qquad
                          \alpha(k) < \alpha(1).
$$
\end{itemize}

Our first result shows that $U_{g+1,n+1}^{III}$ can be inductively constructed via 
unicellular maps of lower genus. 
Let $u=(H, \alpha, \sigma)\in U^{III}_{g+1,n+1}$ with boundary component
$\gamma=(1_R, 1, 2, \ldots, 2n+1, 2n+2, 2(n+1)_R )$. We relabel
$$
\gamma=\left(1_{R_{u_2}},1_{R_{u_1}},1_{u_1},\dots,(2k)_{u_1},
        (2k)_{R_{u_1}},
         1_{u_2},\dots,(2(n-k))_{u_2}, (2(n-k))_{R_{u_2}}\right),
$$
where $1_{R_{u_1}}=1$ and $(2k)_{R_{u_1}}=\alpha(1)$.
\begin{lemma}\label{L:lem1}
There is a bijection
$$
\theta \colon \dot\bigcup_{0\le g_1\le g+1,\; 0\le j\le n}\left(U_{g_1,j}\times  U_{g+1-g_1,n-j}\right)
                        \longrightarrow U^{III}_{g+1,n+1}.
$$
\end{lemma}
\begin{proof}
We begin by specifying the argument for $\theta$:
\begin{itemize}
\item $u_1=(H_1, \alpha_1, \sigma_1)$, with
$$
\gamma_1=(1_{R_{u_1}}, 1_{u_1}, \ldots, (2k)_{u_1}, (2k)_{R_{u_1}})
$$
and the $J_1$ vertices
$v_{1,u_1},\ldots, v_{i,u_1}, \ldots, v_{J_1, u_1}$, where
$$
v_{J_1, u_1}=(2k)_{R_{u_1}},\qquad v_{1, u_1}=(h_{1, u_1}^1, h_{1,u_1}^2, \ldots,
h_{1,u_1}^{m_1}),
$$
for some $m_1>0$, and $h_{1, u_1}^1=1_{R_{u_1}}$ and $h_{1, u_1}^{m_1}=(2k)_{u_1}$.
\item $u_2=(H_2, \alpha_2, \sigma_2)$, with boundary component
$$
\gamma_2=
     (1_{R_{u_2}}, 1_{u_2}, \ldots, (2(n-k))_{u_2}, (2(n-k))_{R_{u_2}})
$$
and the $J_2$ vertices, $v_{1, u_2}
      \ldots v_{i, u_2} \ldots v_{J_2, u_2}$, where
$$
v_{J_2, u_2}=(2(n-k))_{R_{u_2}}, \qquad v_{1,u_2}=
      (h_{1,u_2}^{1}, \ldots, h_{1,u_2}^{m_2}),
$$
for some $m_2>0$ and $h_{1,u_2}^{1}=1_{R_{u_2}}$ and $h_{1, u_2}^{m_2}=(2(n-k))_{u_2}$.
\end{itemize}

For $u_1$ $u_2$, consider the two vertices $v_{J_1, u_1}$ (the plant of $u_1$)
and $v_{1,u_2}$. The key operation consists in ``gluing'' $v_{J_1, u_1}$ into
$v_{1,u_2}$, thereby producing the new vertex
$$
w_{1,u_2}=(h_{1,u_2}^1,(2k)_{R_{u_1}}, h_{1,u_2}^2, \ldots, h_{1, u_2}^{m_2}).
$$
By construction, this produces the unicellular map, $\theta(u_1,u_2)$,
with boundary component
$$
\gamma=\left(1_{R_{u_2}},\underbrace{1_{R_{u_1}},1_{u_1},\dots,(2k)_{u_1},
        (2k)_{R_{u_1}}}_{\gamma_1},
         1_{u_2},\dots,(2(n-k))_{u_2}, (2(n-k))_{R_{u_2}}\right)
$$
and vertex set
$$
\{v_{i,u_1} \mid 1\le i <  J_1\}\dot\cup \{w_{1,u_2}\}\dot\cup
\{v_{i,u_2} \mid 2\le i\le J_2\}.
$$
The combinatorial interpretation of this ``gluing''  is illustrated in
Fig.~\ref{F:proof-glue-slice}.

\begin{figure}[ht]
\begin{center}
\includegraphics[width=0.9 \columnwidth]{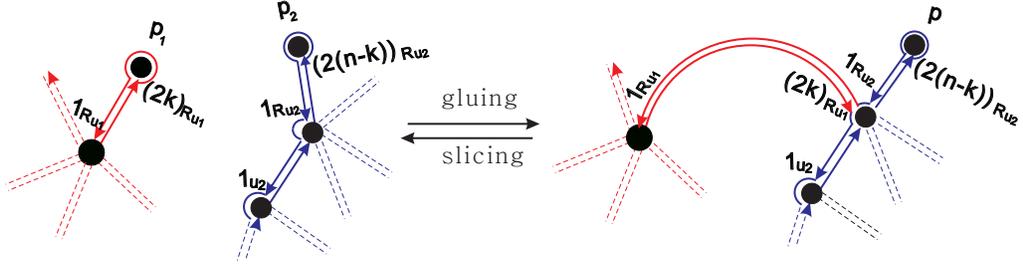}
\end{center}
\caption{\small The "proof" of Lemma~\ref{L:lem1}.}\label{F:proof-glue-slice}
\end{figure}


We next observe that $\theta(u_1,u_2)$ has genus $g_1+g_2=g+1$. Namely, we have
$2-2g_1=(v_{J_1}-1)-k+1$ and $2-2g_2=(v_{J_2}-1)-(n-k)+1$,
i.e.~$2-2(g_1+g_2) = (v_{J_1}+v_{J_2}-2) - (n+1) + 1$. Since the $u_1$-plant
becomes an edge, $\theta(u_1,u_2)$ satisfies
$$
2-2(g+1) = (v_{J_1}+v_{J_2}-2) - (n+1) +1,
$$
whence $\theta(u_1,u_2)$ has genus $g+1=(g_1+g_2)$. 
Consequently we have established the mapping
$$
\theta \colon \dot\bigcup_{0\le g_1\le g+1}\left(U_{g_1,j}\times  U_{g+1-g_1,n-j}\right)
\longrightarrow U^{III}_{g+1,n+1}.
$$

We proceed by showing that $\theta$ is bijective by explicitly specifying its inverse.
To this end let $u=(H, \alpha, \sigma)\in U_{n+1,g+1}^{III}$ with the plant
$p$ and face
$$
\gamma=\left(1_{R_{u_2}},1_{R_{u_1}},1_{u_1},\dots,(2k)_{u_1},
        (2k)_{R_{u_1}},
         1_{u_2},\dots,(2(n-k))_{u_2}, (2(n-k))_{R_{u_2}}\right),
$$
where $1_{R_{u_1}}=1$ and $(2k)_{R_{u_1}}=\alpha(1)$, having $J$ cycles $v_1, v_2,
\ldots, v_J$. By assumption, $\{1_{R_{u_1}},\alpha(1_{R_{u_1}})\}$ is incident to
two different vertices, such that
$$
\forall \;k_{u_1};\  1_{R_{u_1}} < k_{u_1} <  \alpha(1_{R_{u_1}}); \qquad
                          \alpha(k_{u_1}) < \alpha(1_{R_{u_1}}).
$$

We set $v_1=(h_{1}^1, h_{1}^2, h_{1}^3, \ldots, h_{1}^{m})$ with $\min(v_1)=h_1^1=1_{R_{u_2}}$ and
$v_2=(h_{2}^1, h_{2}^2, \ldots, h_{2}^{m^{'}})$ for some $m$ and $m^{'}$ with 
$\min(v_2)=1_{R_{u_1}}$, see Fig.~\ref{F:relation1}.

\begin{figure}[ht]
\begin{center}
\includegraphics[width=0.5\columnwidth]{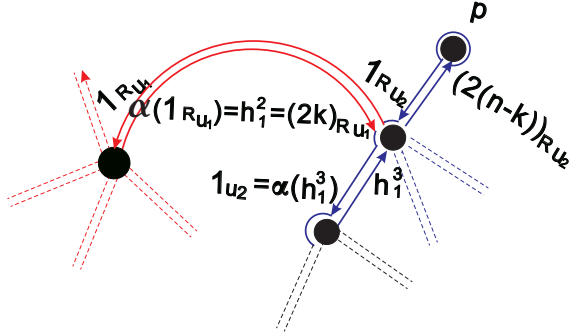}
\end{center}
\caption{\small Constructing the inverse of $\theta$.}\label{F:relation1}
\end{figure}

Then $\alpha(1_{R_{u_1}})=h_{1}^2=(2k)_{R_{u_1}}$, and $\alpha(1_{R_{u_1}})<_u \alpha(h_{1}^3)$.
Since for any $1_{R_{u_1}} < k_{u_1} <\alpha(1_{R_{u_1}})$ we have
$\alpha(k_{u_1}) < \alpha(1_{R_{u_1}})$, the boundary
component $\gamma$ contains a sequence of half-edges
$$
(1_{R_{u_1}},1_{u_1},\dots,(2k)_{u_1},(2k)_{R_{u_1}}).
$$
Let $H_1=\{1_{R_{u_1}},1_{u_1},2_{u_1},\dots,\alpha(1_{R_{u_1}})\}$ and $H_2$ its complement. Then
$\alpha$ induces by restriction two mappings
$$
\alpha|_{H_1} \colon H_1 \longrightarrow  H_1,\qquad
\alpha|_{H_2} \colon H_2 \longrightarrow  H_2.
$$
In particular $\alpha(1_{R_{u_1}})$ is even.

We now introduce the mapping $\psi$ obtained by ``cutting'' the edge
$\{1_{R_{u_1}},\alpha(1_{R_{u_1}})\}$ in $u$.
To this end we introduce the two new vertices
\begin{eqnarray*}
v_{1, u_2}   & = & (h_{1}^1,  h_{1}^3, \ldots, h_{1}^{m}) \\
v_{J_1, u_1} & = & (h_{1}^2).
\end{eqnarray*}

Now, taking the union of $v_{J_1, u_1}$  with all $u$-vertices that are cycles of
half-edges contained in $H_1=\{1_{R_{u_1}},1_{u_1},2_{u_1},\dots,\alpha(1_{R_{u_1}})\}$
and restricting $\alpha$ to $H_1$, generates a new map $u_1$.
$u_1$ is evidently unicellular since $(1_{R_{u_1}},1_{u_1},2_{u_1},\dots,
\alpha(1_{R_{u_1}}))$ is its unique boundary component.

We next replace $v_1$ by $v_{1, u_2}= (h_{1}^1,  h_{1}^3, \ldots, h_{1}^{m})$.
Then the set of all $u$-vertices different from $v_1$ that are contained in
$H_2=H\setminus H_1$ and $v_{1, u_2}$, together with the restriction of $\alpha$ to
$H_2$ form a new map, $u_2$ . The latter is by construction unicellular and
its boundary component is given by
$$
\gamma_2 = (1_{R_{u_2}},1_{u_2},\dots,(2(n-k))_{u_2},(2(n-k))_{R_{u_2}}).
$$
Considering the Euler characteristics we can conclude $(g+1)= g_1+g_2$.

By construction,
$$
\psi\circ \theta=\text{\rm id}_{
\dot\bigcup_{0\le g_1\le g+1,0\le j\le n}\left(U_{g_1,j}\times  U_{g+1-g_1,n-j}\right)}
$$
and
$$
\theta\circ \psi=\text{\rm id}_{U^{III}_{g+1,n+1}},
$$
whence $\theta$ is a bijection.
\end{proof}


The second result relates bicellular maps and unicellular maps of types $I$
and $II$; the key idea is analogous to that in Lemma~\ref{L:lem1}.

Let $B_{g,n}$ denote the set of bicellular map of genus $g$ with $n$ edges.
We observe that $B_{g,n}$ can be written as
$B_{g,n} = B_{g,n}^I\dot\cup\; B_{g,n}^{II}$, where $B_{g,n}^I$ denotes the set of bicellular
maps of genus $g$ with $n$ edges in which the two plants $p_1$ and
$p_2$ are incident to two different vertices and $B_{g,n}^{II}$
denotes its complement. 

Furthermore, for $b\in B_{g,n}^I \dot\cup B_{g,n}^{II}$, we have
\begin{equation}\label{E:b1}
\exists \; x \in \omega_1; \ \beta(x)\in \omega_2.
\end{equation}
Let now $u=(H,\alpha,\sigma)$ be a unicellular map of genus $(g+1)$ having $(n+1)$
edges with boundary component
$$
\gamma=(1_R,1,2,\ldots,2n, (2n+1), (2n+2), (2n+2)_R).
$$
We shall relabel $\gamma$ as
$$
\gamma=\left((m+1)_{R_b},1_{R_b},1_b,\dots, m_b, m_{R_b},
(m+1)_b, \dots, (2n)_b, (2n)_{R_b}  \right).
$$
\begin{lemma}\label{L:lem2}
There exists a bijection
$$
\eta \colon B^{I}_{g, n}\dot\cup \; B^{II}_{g, n}
\longrightarrow U^{I}_{g+1,n+1}\dot\cup \; U^{II}_{g+1,n+1},
$$
and $\eta$ induces by restriction the two bijections
$$
\eta_I \colon B^{I}_{g, n} \longrightarrow U^{I}_{g+1,n+1}\quad\text{\rm and}\quad
\eta_{II} \colon B^{II}_{g, n} \longrightarrow U^{II}_{g+1,n+1}.
$$
\end{lemma}

\begin{figure}[ht]
\centerline{%
\epsfig{file=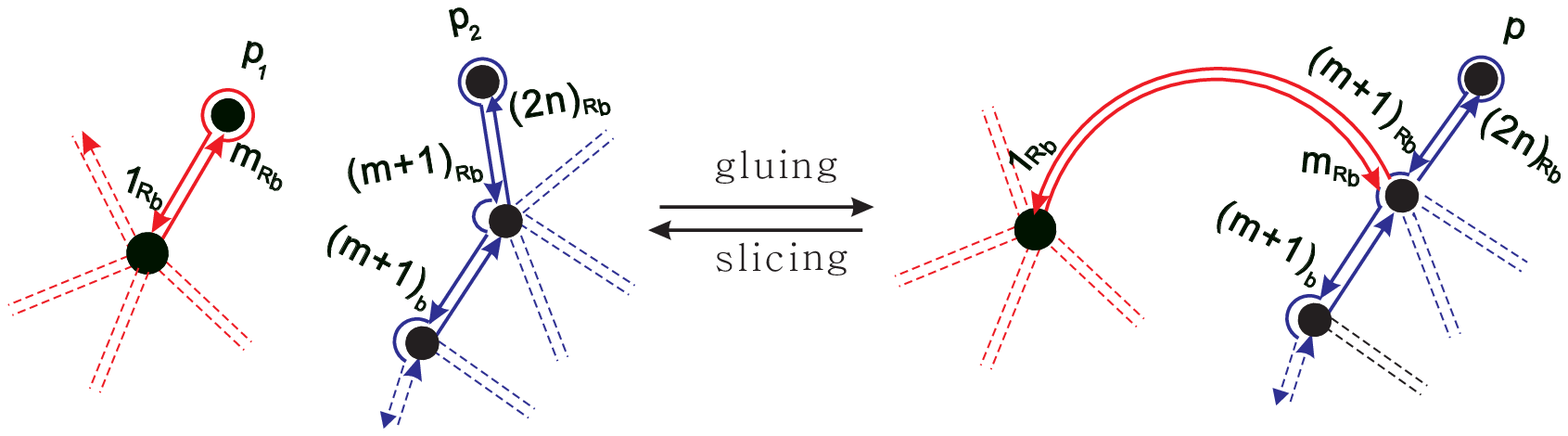,width=0.9 \textwidth}\hskip15pt }
\caption{\small Lemma~\ref{L:lem2}: gluing and slicing, the case $B^{I}_{g, n+1}$.
}\label{F:type-b}
\end{figure}

\begin{figure}[ht]
\centerline{%
\epsfig{file=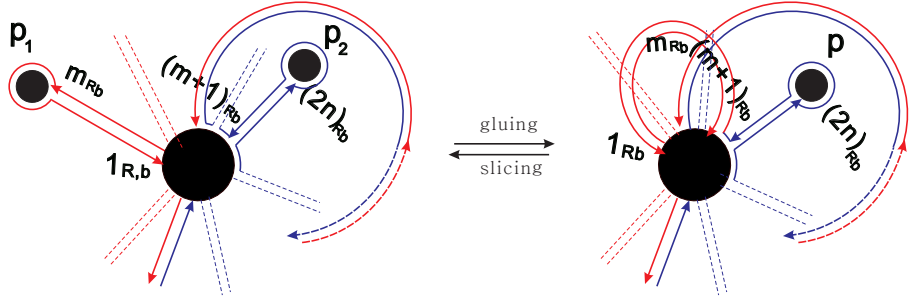,width=0.8 \textwidth}\hskip15pt }
\caption{\small Lemma~\ref{L:lem2}: gluing and slicing, the case $B^{II}_{g, n+1}$.
}\label{F:type-c}
\end{figure}

\begin{proof}
Let $b=(L,\beta,\tau)$ be a planted bicellular map of genus $g$ having $n$ edges with
plants $p_1$ and $p_2$ and tour $\beta \circ \tau=\omega_1\circ \omega_2$. Let
$$
\omega_1=(1_{R_b}, 1_b, \ldots, m_b, m_{R_b}) \quad\text{\rm and}\quad
\omega_2=((m+1)_{R_b}, (m+1)_b, \ldots, 2n_b, 2n_{R_b})
$$
and $v_{i, b}$, for $1\leq i \leq J$ be the set of its vertices.

Consider the two vertices $(m_{R_b})$ and  $v_{p_2, b}$, where $(m_{R_b})$ is
the plant $p_1$ and $v_{p_2, b}$ denote the cycle containing half-edge 
$(m+1)_{R_b}$. I.e.~we have
$$
v_{p_2, b}= (h^1_{p_2, b}, h^2_{p_2, b}, \ldots, h^l_{p_2, b}), \quad
\text{for \quad} l > 0.
$$
where $h^1_{p_2, b}=\min v_{p_2, b}$.
Note that if $h^1_{p_2, b}= \min v_{p_2, b} = (m+1)_{R_b}$, then  $v_{p_2, b}$ is
different from $v_{1,b}=(1_{R_b},\dots)$, whence $b \in B_{g,n}^{I}$ and if
$h^1_{p_2, b}= 1_{R_b}$, then $v_{p_2, b}=v_{1,b}$ and consequently $b \in B_{g,n}^{II}$.

The key operation consists in "gluing " $p_1=(m_{R_b})$ into $v_{p_2, b}$. This generates
the unicellular map, $\eta(b)$, with boundary component
\begin{equation}\label{E:bc-eta(b)}
\gamma=\left((m+1)_{R_b},\underbrace{1_{R_b},1_b,\dots, m_b, m_{R_b}}_{\omega_1}, (m+1)_b,
\dots, (2n)_b, (2n)_{R_b}  \right)
\end{equation}
and the new vertex
$$
w_{p_2, b}=(h_{p_2, b}^1,  m_{R_b}, h_{p_2, b}^2, \ldots,
h^l_{p_2, b}) \quad \text{where} \quad h^1_{p_2, b}= (m+1)_{R_b},
$$
obtained by gluing $(m_{R_b})$ into $v_{p_2, b}$. Accordingly, $\eta(b)$ has vertex set
$$
(\{v_{i,b} \mid 1\leq i \leq J\} \dot\cup  \{w_{p_2, b}\})\setminus\{v_{p_2, b}, (m_{R_b})\}.
$$
Note that in case of $v_{1,b} \neq v_{p_2, b}$, the gluing does not merge these 
$b$-vertices.

Suppose now $b \in B^{I}_{g, n}$. By definition there exists some 
$x \in \omega_1$ such that $\beta(x)\in \omega_2$. 
Thus there exists some $1_{R_b}< x < m_{R_b}$, such that $(m+1)_{R_b}<\beta(x)<2n_{R_b}$.
Gluing produces the unicellular map $\eta(b)$ whose boundary component is
given in eq.~\ref{E:bc-eta(b)}. We observe that the half-edges $1_{R_b}$ and $m_{R_b}$ 
map exactly to the half-edges $1$ and $\alpha(1)$ in $\eta(b)$ and furthermore, 
the half-edges $(m+1)_b$ to $2n_{R_b}$ map to $\eta(b)$-halfedges that are all 
greater than $\alpha(1)$. Consequently, there exists some $\eta(b)$ half-edge, 
$1< x <\alpha(1)$ such that $\alpha(x) >\alpha(1)$, whence $\eta(b) \in U^{I}_{g+1, n+1}$.

The case of $b \in B^{II}_{g, n}$ is analogous. Then there also exists some $1< x 
<\alpha(1)$ such that $\alpha(x) >\alpha(1)$ holds, whence $\eta(b) \in U^{II}_{g+1, n+1}$.
In Fig.~\ref{F:type-b} and Fig.~\ref{F:type-c} we depict what happens if
we glue $(m_{R_b})$ into $v_{p_2, b}$ in these respective cases.

We next inspect that $\eta(b)$ has genus $(g+1)$. Indeed, $b$ satisfies
$2-2g=v_{J}-n+2$, i.e.~$2-2(g+1)=v_{J}-(n+1)+1$. Since the gluing transforms 
the $p_1$ plant into an edge, the latter equation shows that $\eta(b)$ has genus $(g+1)$.

We have thus shown that there exists a welldefined mapping
$$
\eta \colon B^{I}_{g, n}\dot\cup \; B^{II}_{g, n}
\longrightarrow U^{I}_{g+1,n+1}\dot\cup \; U^{II}_{g+1,n+1},
$$
that induces by restriction the mappings
$$
\eta_I \colon B^{I}_{g, n} \longrightarrow U^{I}_{g+1,n+1}\quad\text{\rm and}\quad
\eta_{II} \colon B^{II}_{g, n} \longrightarrow U^{II}_{g+1,n+1}.
$$


We next construct the inverse of $\eta$.
To this end, let $u=(H, \alpha, \sigma)$ be a unicellular map of genus $g$ having 
$(n+1)$ edges with plant $p$ and face
\begin{equation}\label{E:bc-u}
\gamma=\left( (m+1)_{R_b}, 1_{R_b}, 1_b, \dots, m_b, m_{R_b}, (m+1)_b, \dots, (2n)_b, (2n)_{R_b}
       \right),
\end{equation}
where $1_{R_{b}}=1_u$, $\alpha(1_u)= m_{R_b} $ and $\sigma$ having the 
$K$ cycles $v_{1, u}, v_{2, u},\ldots, v_{K, u}$.

\begin{figure}[ht]
\begin{center}
\includegraphics[width=0.7\columnwidth]{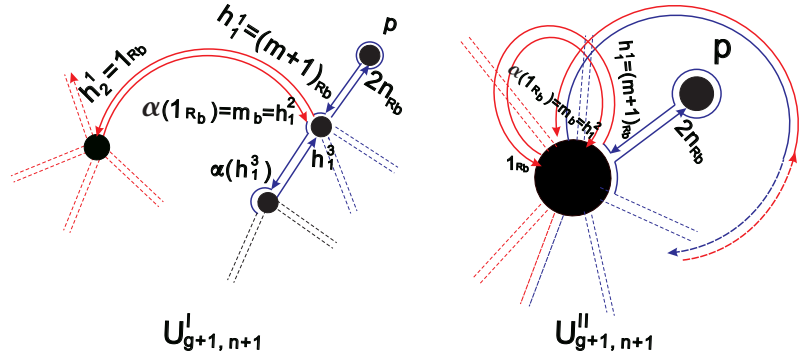}
\end{center}
\caption{\small Constructing the inverse of $\eta$.}\label{F:relation2}
\end{figure}

Suppose first $u \in U_{g+1,n+1}^{I}$, then $\{1_{R_{b}},\alpha(1_{R_{b}})\}$ is 
incident to two different vertices. We set
$$
v_1=(h_{1}^1, h_{1}^2, h_{1}^3, \ldots, h_{1}^{j})
$$
where $h_1^1=(m+1)_{R_b}$ and
$$
v_{2}=(h_{2}^1, h_{2}^2, \ldots, h_{2}^l),
$$
where by definition $h_{2}^1=1_{R_{b}}$ and $\alpha(1_{R_{b}})=m_{R_b}$.
Consequently, we have
$$
h_1^2=\sigma(h_1^1)= (\alpha \circ \gamma) (h_1^1)= \alpha (1_{R_b})=m_{R_b}.
$$
Second let $u \in U^{II}_{g+1, n+1}$. Then $\{1_{R_{b}},\alpha(1_{R_{b}})\}$
is incident to
$$
v_1=(h_{1}^1, h_{1}^2, h_{1}^3, \ldots, h_{1}^{j}),
$$
such that $h_1^1=(m+1)_{R_{b}}$ and
$$
h_1^2= (\alpha \circ \gamma)(h_1^1 ) = \alpha(1_{R_b}) = m_{R_b}.
$$
Since $\alpha(1_{R_{b}})$ and $1_{R_{b}}$ are incident to the same vertex and 
$h_1^2=\alpha(1_{R_b})$, we can conclude that $1_{R_{b}}$ is a half-edge of $v_1$.
In other words, there exist some half-edge $h_{1}^{a}$ (namely $h_{1}^{a}=1_{R_b}$), for
some $1< a <j$, such that $\alpha(h_{1}^a)=m_{R_{b}}$.

We now introduce the mapping $\varsigma$ obtained by ``cutting'' the edge $\{1_{R_b},
\alpha(1_{R_b})\}$ in $u$. That is, $\varsigma$ splices $v_1$ into the two vertices
\begin{eqnarray*}
v_{p_2, b}   & = & (h_{1}^1,  h_{1}^3, \ldots, h_{1}^{j}) \\
p_1 & = & (h_{1}^2).
\end{eqnarray*}
This process generates the new vertex set
$$
\{v_{i} \mid 1< i \leq K \}\cup \{v_{p_2, b}\} \cup \{ p_1\}.
$$
Since $\alpha(1_{R_{b}})=m_{R_b}$, the sequence of half-edges
$$
\omega_1=(1_{R_b},1_b,2_{b},\dots,\alpha(1_{R_{b}}))
$$
and
$$
\omega_2=\left( (m+1)_{R_b}, (m+1)_b, \dots, (2n)_b, (2n)_{R_b}\right),
$$
represent the two boundary components of the new map. \\
Furthermore, since for $u \in U^{I}_{g+1,n+1}\dot\cup \; U^{II}_{g+1,n+1}$, 
there exists some $k$, $1 <  k < \alpha(1)$, such that $\alpha(1) < \alpha(k)$. 
Since $u$ has the boundary component given in eq.~(\ref{E:bc-u}), we have 
$1=1_{R_b}$ and $\alpha(1)=m_b$. Accordingly, there exists some
$1_{R_{b}} <  k < \alpha(1_{R_{b}})$, such that $\alpha(1_{R_{b}}) < \alpha(k)$.

$\varsigma(u)$ has the two new boundary components  $\omega_1$ and $\omega_2$, 
i.e.~there exist some $k \in \omega_1$, such that $\beta(k) \in \omega_2$ and
$\varsigma(u)$ is bicellular map. By construction, if $u \in U^{I}_{g+1,n+1}$,
then $\varsigma(u)$ has its two plants incident to two distinct vertices, whence
$ \varsigma(u)\in B^{I}_{g,n}$. In case of $u \in U^{II}_{g+1,n+1}$ then
$\varsigma(u)$ has its two plants incident to one vertex and $\varsigma(u) \in
B^{II}_{g,n}$.

Euler's characteristic formula implies that $g(u) = g(\varsigma(u))+1$. 
Furthermore, by construction,
$$
\varsigma\circ \eta =\text{\rm id}_{B_{g,n}} \quad\text{\rm and} \quad
\eta\circ \varsigma =\text{\rm id}_{{U^{I}_{g+1,n+1}} \dot\cup {U^{II}_{g+1,n+1}}}.
$$
Thus $\eta$ is a bijection that induces by construction the bijections
$\eta_I$ and $\eta_{II}$ and the lemma follows.
\end{proof}



\section{The main result}


\begin{theorem}\label{T:thm1}
Let $U_{g,n}$ and $B_{g,n}$ denote the sets of unicellular and bicellular maps
containing $n$ edges and genus $g$. Then there is a bijection
\begin{equation}
\beta\colon \dot\bigcup_{0\le g_1\le g+1,\; 0\le j\le n}\left(U_{g_1,j}\times U_{g+1-g_1,n-j}\right)
\;\dot \cup \; B_{g,n}  \longrightarrow U_{g+1, n+1}.
\end{equation}
\end{theorem}
\begin{proof}
We have shown in Lemma~\ref{L:lem1} and Lemma~\ref{L:lem2} that there are bijections
\begin{eqnarray*}
\theta \colon \dot\bigcup_{0\le g_1\le g+1,\; 0\le j\le n+1}\left(U_{g_1,j}\times
U_{g+1-g_1,n-j}\right) & \longrightarrow &  U^{III}_{g+1,n+1} \\
\eta \colon B_{g, n+1}
 & \longrightarrow & U^{I}_{g+1,n+1}\dot\cup \; U^{II}_{g+1,n+1}.
\end{eqnarray*}
These two bijections determine
$$
\beta\colon \dot\bigcup_{g_1, j}\left(U_{g_1,j}\times  U_{g+1-g_1,n-j}\right)
\;\dot \cup \; B_{g,n}  \longrightarrow U_{g+1, n+1},
$$
whence the theorem.
\end{proof}

An immediate enumerative corollary of this bijection is the following result:

\begin{corollary}\label{C:2to1}
The generating function of unicellular and bicellular maps,
${\bf C}_{g}(z)$ and ${\bf C}_{g}^{[2]}(z)$ satisfy the following equation
\begin{equation}
\sum_{g_1=0}^{g+1}\,{\bf C}_{g_1}(z){\bf C}_{g+1-g_1}(z)
+{\bf C}_{g}^{[2]}(z)= {\bf C}_{g+1}(z)/z,
\end{equation}
which equivalent to the coefficient equation
\begin{equation}\label{E:recursion}
\sum_{g_1=0}^{g+1}\,\sum_{i \geq 0}^n\,{\sf c}_{g_1}(i) {\sf c}_{g+1-g_1}(n-i)
+ {\sf c}_g^{[2]}(n)={\sf c}_{g+1}(n+1)
\end{equation}
\end{corollary}

Corollary~\ref{C:2to1} can be proved using either (a) the matrix model 
\cite{Dyson, Schwinger} and in particular the Schwinger-Dyson equation or 
(b) representation theory \cite{Sagan,Zagier}.


\end{document}